\documentclass[12pt]{amsart}

\advance\hoffset by -0.5 truecm

\usepackage{amssymb}
\usepackage{amscd}
\usepackage{verbatim}
\usepackage{epsfig}
\usepackage{euscript}
\usepackage[colorlinks=true]{hyperref}
\hypersetup{urlcolor=blue, citecolor=red}

\newcommand{\supp}{\operatorname{supp}}

\newcommand{\WF}{\operatorname{WF}}

\newcommand{\diag}{\mathrm{diag}}

\newcommand{\Ker}{\operatorname{Ker}}

\def \R{\mathbb R}
\def \p{\partial}

\def \<{\langle}
\def \>{\rangle}

\def \1N{\sum_{k=1}^{N}}
\def \A{\mathcal A}
\def \B{\mathcal B}

\def \C{\mathbb C}

\def \bfL{{\bf L}^2}
\def \bfH{{\bf H}}
\def \dgamma{\dot\gamma}
\newcommand{\id}{\operatorname{id}}

\newtheorem{lemma}{Lemma}[section]
\newtheorem{prop}[lemma]{Proposition}
\newtheorem{thm}[lemma]{Theorem}

\newtheorem*{thm*}{Theorem}
\newtheorem*{prop*}{Proposition}
\newtheorem*{cor*}{Corollary}
\newtheorem*{conj*}{Conjecture}
\numberwithin{equation}{section}
\theoremstyle{remark}
\newtheorem{rem}[lemma]{Remark}
\newtheorem*{rem*}{Remark}
\theoremstyle{definition}

\newtheorem*{Def*}{Definition}

\begin{document}

\title[Inverse problems for generic connections]{Generic injectivity and stability of inverse problems for connections}

\author[Hanming Zhou]{Hanming Zhou}
\address{Department of Pure Mathematics and Mathematical Statistics, University of Cambridge, Cambridge CB3 0WB, UK}
\email{hz318@dpmms.cam.ac.uk}

\begin{abstract}
We consider the nonlinear problem of determining a connection and a Higgs field from the corresponding parallel transport along geodesics on a compact Riemannian manifold with boundary, in any dimension. The problem can be reduced to an integral geometry question of some attenuated geodesic ray transform through a pseudolinearization argument. We show injectivity (up to natural obstructions) and stability estimates for both the linear and nonlinear problems for generic simple metrics and generic connections and Higgs fields, including the real-analytic ones. We consider the problems on simple manifolds in order to make the exposition of the main ideas clear and concise, many results of this paper are still true under much weaker geometric assumptions, in particular conjugate points and trapped geodesics are allowed and the boundary is not necessarily convex.
\end{abstract}

\maketitle

\section{Introduction} \label{sec_introduction}

Let $(M,g)$ be a compact Riemannian manifold with smooth boundary $\p M$, $n=\dim M\geq 2$. Let $A$ be a connection on the trivial bundle $M\times \mathbb C^k$ of rank $k$, which simply means that $A$ is a $k\times k$ matrix whose entries are 1-forms on $M$ with complex values. We also introduce a Higgs field $\Phi\in C^\infty(M;\C^{k\times k})$, a complex matrix function on $M$, and denote the pair $(A,\Phi)$ by $\A$. We define the {\it parallel transport} associated with $\A$ of a vector $u_0\in \mathbb C^k$ along a geodesic $\gamma:[0,T]\to M$, $\gamma(0),\gamma(T)\in \p M$, as the solution of the following ODE
\begin{equation}\label{parallel transport}
\dot u+\A(\gamma,\dot\gamma)u=0,\quad u(0)=u_0.
\end{equation}
Here $\A(\gamma,\dgamma)=A_{\gamma}(\dgamma)+\Phi(\gamma)$. In the mean time, there is a fundamental matrix solution $U:[0,T]\to GL(k,\mathbb C)$ of \eqref{parallel transport} which satisfies 
\begin{equation}\label{fundamental matrix}
\dot U+\A(\gamma,\dgamma)U=0,\quad U(0)=\id.
\end{equation}
It is easy to see that $u(t)=U(t)u_0$, thus the information of the parallel transport is encoded in the fundamental matrix $U$. We are interested in the inverse problem of recovering the pair $(A,\Phi)$ on $M$ from the information of the parallel transport at the end point, i.e. $U(T)$, given there are enough geodesics $\gamma$ covering the manifold.

To make the exposition of the main ideas clear and concise, in this paper we assume that $(M,g)$ is a {\it simple} manifold, which means that $\p M$ is strictly convex and the exponential map is a diffeomorphism at any point $x\in M$. In the mean time, we can always assume that $(M,\p M)$ is equipped with a real-analytic atlas (the metric $g$ may not be real-analytic).

Let $SM$ be the unit sphere bundle of $M$ and $\p SM$ be its boundary, we define two subsets of $\p SM$
$$\p_\pm SM:=\{(x,v)\in \p SM: \,\pm\<v,\nu(x)\>_g\geq 0\},$$
where $\nu(x)$ is the unit inward normal vector to $\p M$ at $x$. Given $(x,v)\in SM$, we denote $\gamma_{x,v}$ the unique maximal geodesic on $M$ satisfying $\gamma_{x,v}(0)=x,\,\dgamma_{x,v}(0)=v$, let $\tau(x,v)$ ($\tau_-(x,v)$) be the positive (negative) time the geodesic $\gamma_{x,v}$ exits $M$. $\tau$ and $\tau_-$ are smooth in $SM\setminus S(\p M)$ and continuous on $SM$. 
On the other hand, for any $(x,v)\in SM$, there exists unique $(x_0,v_0)\in \p_+SM$ such that $\gamma_{x_0,v_0}(-\tau_-(x,v))=(x,v)$. Thus one can define $U_{\A}:SM\to GL(k,\mathbb C)$ by 
$$U_{\A}(x,v)=U(-\tau_-(x,v)),$$
where $U$ is the fundamental matrix solution of \eqref{fundamental matrix} along $\gamma_{x_0,v_0}$. $U_{\A}$ satisfies the following transport equation 
$$XU_{\A}+\A U_{\A}=0,\quad U_{\A}|_{\p_+SM}=\id,$$
where $X$ is the generating vector field of the geodesic flow. It is easy to see that $U_\A$ has the same regularity as $\tau$. 

Now we can define the {\it scattering data} associated with $\A$
$$C_\A: \p_+SM\to GL(k,\mathbb C)$$
by $C_\A(x,v):=U_\A(\gamma_{x,v}(\tau(x,v)),\dgamma_{x,v}(\tau(x,v)))$, or $C_\A=U_\A|_{\p_-SM}$ in short. Our first result is regarding the recovery of $\A=(A,\Phi)$ from $C_\A$. Notice that there is a natural gauge of this problem: let $p: M\to GL(k,\mathbb C)$ with $p|_{\p M}=\id$, then $C_{A,\,\Phi}=C_{p^{-1}dp+p^{-1}A p,\, p^{-1}\Phi p}$. Define $d_\A p:=[(d+A)p,\Phi p]$, then the equality just means that $C_\A=C_{p^{-1}d_\A p}$. Thus one can only expect to determine $\A$ up to the gauge. 

\begin{thm}\label{nonlinear thm}
Let $M$ be a real-analytic simple manifold with real-analytic metric $g_0$. Let $\A_0,\, \B_0$ be real-analytic, there exists $\epsilon>0$ such that whenever there are another metric $g$ and pairs $\A=(A,\Phi)$, $\B=(B,\Psi)$ satisfying 
$$\|g-g_0\|_{C^4(M)}\leq \epsilon,\quad \|\A-\A_0\|_{C^3(M)}+\|\B-\B_0\|_{C^3(M)}\leq \epsilon,$$
(1) if $C_\A=C_\B$ w.r.t. the metric $g$, then there is $p:M\to GL(k,\mathbb C)$ with $p|_{\p M}=\id$, such that $\B=p^{-1}d_\A p$;\\
(2) if $\|\A_0-\B_0\|_{C^2(M)}\leq \epsilon$ and $\iota^*\A=\iota^*\B$ with $\iota:\p M\to M$ the canonical inclusion, then there exists $p:M\to GL(k,\mathbb C)$ with $p|_{\p M}=\id$ such that the following stability estimate holds w.r.t. the metric $g$
$$\|\B -p^{-1}d_\A p \|_{\bfL(M)}\leq C \|C_\B-C_\A\|_{H^1(\p_+SM)}$$
for some uniform constant $C>0$ which depends only on $g_0,\,\A_0,\,\B_0$.
\end{thm}
Notice that $\A_0$ is complex-valued, we say that $\A_0$ is real-analytic if both the real and imaginary parts of $\A_0$ are real-analytic. $\|\cdot\|_{\bfH^k}$ is the natural $H^k$ norm for pairs, $k\geq 0$, see Section 2 for the definition.

Theorem \ref{nonlinear thm} shows that the rigidity result (up to the natural gauge) hold for generic simple metrics and generic connections and Higgs fields, including the real-analytic ones. There are previous works on the determination of connections from the parallel transport along straight lines in the Euclidean spaces \cite{Ve92,FU01,No02,Es04}. Injectivity results are valid on simple surfaces \cite{PSU12}, simple manifolds for connections which are $C^1$ close to a given one with small curvature \cite{Sh00} and negatively curved manifolds with strictly convex boundary \cite{GPSU16}. Though \cite{GPSU16} allows the existence of trapped geodesics, above references on general manifolds all require the connections (and Higgs fields) to be unitary. The only exception is \cite{PSUZ16} which considers manifolds of dimension $\geq 3$ with strictly convex boundary that admits a strictly convex function, in particular the last assumption is true if the manifold has non-negative sectional curvatures. In the current paper, we put no restrictions on the connections and Higgs fields, the dimension or the curvatures, and the simplicity assumption indeed can be much weakened, see Remark \ref{microlocal condition}. In particular our method also applies on simple surfaces ($n=2$) to non-unitary connections. 

To prove Theorem \ref{nonlinear thm}, which is regarding a nonlinear rigidity problem, we will reduce it to an integral geometry problem through a ``linerization" of the scattering data, which is inspired by the idea of \cite{SU98} and already appeared in e.g. \cite{PSU12,PSUZ16}. In particular, this motivates us to consider some type of weighted geodesic ray transforms.

Notice that the inverse of $U_\A$, denoted by $W_\A$, satisfies 
$$XW_\A=W_\A \A, \quad W_\A|_{\p_+SM}=\id.$$
Given $\alpha\in C^\infty(T^*M,\mathbb C^k)$ and $f\in C^\infty(M,\mathbb C^k)$, we consider the following geodesic ray transform along $\gamma_{x,v},\,(x,v)\in \p_+SM$
\begin{align*}
I_\A [\alpha,f](x,v)=\int_0^{\tau(x,v)} W_\A(\gamma_{x,v},\dot\gamma_{x,v})\Big(\alpha_{\gamma_{x,v}}(\dot\gamma_{x,v})+f(\gamma_{x,v})\Big)\,dt
\end{align*}
So $I_\A$ is an attenuated geodesic ray transform with attenuation $\A$. The natural elements of the kernel of $I_\A$ are $d_\A p$ with $p\in C^\infty(M,\C^k)$, $p|_{\p M}=0$. If they consist of the whole kernel, then we say $I_\A$ is {\it s-injective}. When $\A=0$, i.e. $W_\A=\id$, the question is reduced to the injectivity of the usual (unweighted) geodesic ray transform of functions or tensor fields (known as the {\it tensor tomography} problem), which has been extensively studied. The geodesic ray transform of functions \cite{Mu77,MR78} and 1-forms \cite{AR97} are s-injective on simple manifolds. See \cite{PSU13,PSU15} and the survey \cite{PSU14} for recent developments of the tensor tomography problem on simple manifolds. Much less is known for the case with attenuations, the question of the s-injectivity of $I_\A$ is still open on simple manifolds. Some partial answers to this question can be found in e.g. \cite{SaU11,PSU12,GPSU16}. It is also worth mentioning that recently tools from microlocal analysis lead to several new local and global results \cite{UV16,SUV14,Gu16,PSUZ16}.

If one restricts the objects in the real-analytic category, there is another approach by applying the analytic microlocal analysis which was initiated in \cite{SU05} by Stefanov and Uhlmann, and further developed in \cite{SU08} for the ordinary tensor tomography problem. The next theorem, which can be viewed as a generalization, shows that $I_\A$ is s-injective for real-analytic simple metric $g$ and real-analytic $\A$ in any dimension.

\begin{thm}\label{analytic s-injective}
Let $M$ be a real-analytic simple manifold with real-analytic metric $g$, let $\A$ be real-analytic, then $I_\A$ is s-injective.
\end{thm}

\begin{rem*}
For the sake of simplicity, we carry out all the arguments with the original complex-valued $\A$ and $W_\A$. Indeed one can reduce everything to real-valued objects and consider an equivalent problem in the real category, see Appendix \ref{complex to real}. 
\end{rem*}

We also remark that there are related studies in the analytic category of weighted X-ray transforms in \cite{FSU08,HS10,Ho13}, they either only consider the function case or impose extra conditions on the 1-forms which make the kernel of the ray transform trivial and the arguments simpler too. 

Similar to Theorem \ref{nonlinear thm}, we also get generic s-injectivity and stability estimates for $I_\A$ by investigating some normal operator involving $I_\A$ through microlocal analysis. The method goes back to the study of the stability estimates of the geodesic ray transform of tensor fields by Stefanov and Uhlmann \cite{SU04,SU05}.  
To state the results, we need to make extensions of the manifold $(M,g)$ and $\A$. Let $M_1$ be a slightly larger compact manifold with boundary so that $M\Subset M_1^o$, where $M_1^o$ is the interior of $M_1$. In particular, one can consider $M_1$ as $M\cup(\p M\times [-\varepsilon,0))$ with $\p M\times [-\varepsilon,0)$ a thin annulus around $M$ for $0<\varepsilon\ll 1$. We also extend $g$ and $\A$ continuously (e.g. under H\"older norms) onto $M_1$ so that $M_1$ is simple too. We can keep $M_1$ being equipped with real-analytic atlas too, and the extended $g$ and $\A$ are real-analytic if the original ones are real-analytic.

Let $\tilde U_\A$ be the fundamental matrix on $SM_1$
$$X\tilde U_\A+\A \tilde U_\A=0,\quad \tilde U_\A|_{\p_+SM_1}=\id.$$
Similarly we denote the inverse of $\tilde U_\A$ by $\tilde W_\A$. We extend the pair $[\alpha,f]$ by zero onto $M_1$ and consider the new ray transform associated with the extended system
$$\tilde I_\A [\alpha,f](x',v')=\int \tilde W_\A(\gamma_{x',v'},\dot\gamma_{x',v'})\Big(\alpha_{\gamma_{x',v'}}(\dot\gamma_{x',v'})+f(\gamma_{x',v'})\Big)\,dt$$
for $(x',v')\in \p_+SM_1$. Given $(x,v)\in \p_+SM$ there exist $t>0$ and $(x',v')\in \p_+SM_1$ such that $(x,v)=(\gamma_{x',v'}(t),\dgamma_{x',v'}(t))$, generally $I_\A [\alpha,f](x,v)\neq \tilde I_\A [\alpha,f](x',v')$. However we will show in Section 2 that one can manipulate the difference, thus knowing $I_\A [\alpha,f]$ is equivalent to knowing $\tilde I_\A [\alpha,f]$. 

Let $\tilde I^*_\A$ be the adjoint of $\tilde I_\A$ under the $L^2$ inner product, we define the normal operator on $M_1$
$$N_\A:=\tilde I^*_\A\, \tilde I_\A.$$
We denote $[\alpha,f]$ by $h$, there exists a unique orthogonal decomposition (w.r.t. $L^2$ inner product) of $h$ on $M$ 
$$h=h^s_M+d_\A p,$$
where $p\in C^{\infty}(M,\mathbb C^k)$, $p|_{\p M}=0$ and $\delta_\A h^s_M=0$ on $M$ (we use $h$ to denote both a function on $M$ and its extension by zero on $M_1$). Here $\delta_\A$ is the adjoint of $d_\A$ under the $L^2$ inner product. See Section 2 and Appendix \ref{decomposition} for more details.

\begin{thm}\label{linear stability}
Let $(M,g)$ be a simple manifold and $\A$ be a pair $[A,\Phi]$, assume that $I_\A$ is s-injective,\\
(1) let $h=[\alpha,f]$, then the following stability estimate for $N_\A$ holds
$$\|h^s_M\|_{\bfL(M)}\leq C \|N_\A h\|_{\bfH^1(M_1)};$$
(2) there exists $0<\epsilon\ll 1$ such that the estimate in (1) remains true if $g$ and $\A$ are replaced by $\tilde g$ and $\tilde A$ satisfying $\|\tilde g-g\|_{C^4(M_1)}\leq \epsilon$, $\|\tilde \A-\A\|_{C^3(M_1)}\leq \epsilon$. The constant $C>0$ can be chosen uniformly, only depending on $g,\,\A$.
\end{thm}

It is easy to see that Theorem \ref{analytic s-injective} and \ref{linear stability} together imply that the s-injectivity of $I_\A$ and the stability estimates hold for generic simple metrics and generic connections and Higgs fields.

\begin{rem}\label{microlocal condition}
Several results of this paper, especially the results of the linear problem, will still hold on a compact manifold satisfying some microlocal condition which essentially says that the union of the conormal bundles of nontrapped geodesics that are free of conjugate points covers the cotangent bundle $T^*M$. This condition allows the existence of trapped geodesics and conjugate points, so one only has access to partial data, and the boundary is not necessarily convex, see \cite{SU08, FSU08} and Remark \ref{generalize 1}, \ref{generalize 2}, \ref{generalize 3} for more details.
\end{rem}

\begin{rem}
We just consider the inverse problem on ordinary geodesics in this paper, however the results can be generalized to general smooth curves, even with nonconstant speed, see previous studies \cite{FSU08, HS10} and the local problem \cite[Appendix]{UV16}.
\end{rem}

\begin{rem}
The arguments of this paper also work for X-ray transforms of vector-valued functions with smooth invertible matrix weights, see Section \ref{function case} for the statements of the results. The scalar case was considered in \cite{FSU08}, and a version for Radon transforms was studied in \cite{HZ16}. Investigations of some related local problems in dimension $\geq 3$ can be found in e.g. \cite{SUV16,Zh16,PSUZ16}. 
\end{rem}

The paper is organized as follows. Section 2 discusses the necessary properties of $I_\A$ for carrying out the arguments of the paper. We prove Theorem \ref{linear stability} in Section 3 and Theorem \ref{analytic s-injective} in Section 4. In Section 5, we discuss the analogous results for weighted ray transforms of functions on $M$, there is no natural gauge in this case. The proof of Theorem \ref{nonlinear thm} is given in Section 6. There are two appendices at the end: Appendix \ref{complex to real} shows that one can reduce everything from complex to real; Appendix \ref{decomposition} establishes an orthogonal decomposition of pairs of functions and 1-forms with respect to $\A$.

\bigskip

\noindent {\bf Acknowledgements.} 
The author thanks Prof. Gunther Uhlmann for suggesting this problem and useful comments. He is also grateful to Prof. Gabriel P. Paternain for very helpful discussions and suggestions on the paper. The research was supported by EPSRC grant EP/M023842/1.



\section{Preliminaries}


Consider $h=[\alpha,f]$ as an element of the space $\bfH^k(M)$, $k\geq 0$, with the norm 
$$\|h\|^2_{\bfH^k(M)}:=\|\alpha\|^2_{H^k(M)}+\|f\|^2_{H^k(M)}.$$
By Theorem \ref{decomposition thm}, there is a unique orthogonal decomposition of $h$ with the form 
$$h=h^s+d_\A p$$
for some $h^s\in \bfH^k(M)$ and $p\in H^{k+1}(M)$ with $p|_{\p M}=0$. Recall that $d_\A p=[dp+Ap,\Phi p]$, let $\delta_\A$ be the adjoint of $d_\A$ under the $L^2$ inner product, then $\delta_\A h^s=0$. Note that $\delta_\A[\alpha,f]=\delta\alpha+A^*(\alpha)+\Phi^*f$. We call $h^s$ and $d_\A p$ the {\it solenoidal} and {\it potential} part of $h$ respectively.

Denote $\Delta_\A=\delta_\A d_\A$, it is easy to see that $\Delta_\A$ is an elliptic operator. One can check that $p$ solves the following regular elliptic Dirichlet boundary value problem
$$\Delta_\A p=\delta_\A h,\quad p|_{\p M}=0.$$
We denote the solution operator, which is the Dirichlet realization of $\Delta_\A$ on $M$, by $\Delta_\A^D$, thus $p=(\Delta^D_\A)^{-1}\delta_\A h$. Define two projections
$$\mathcal P_\A:=d_\A(\Delta^D_\A)^{-1}\delta_\A,\quad \mathcal S_\A=Id-\mathcal P_\A,$$
then $h^s=\mathcal S_\A h$. One can check that $N_\A\mathcal S_\A=\mathcal S_\A N_\A=N_\A$ and $N_\A \mathcal P_\A=\mathcal P_\A N_\A=0$. If we denote $\mathcal S_\A \bfH^k(M)$ and $\mathcal P_\A \bfH^k(M)$ the subspaces of solenoidal and potential pairs (w.r.t. $\A$) of $\bfH^k(M)$ respectively, then obviously 
$$\mathcal S_\A:\bfH^k(M)\to \mathcal S_\A \bfH^k(M),\quad \mathcal P_\A:\bfH^k(M)\to \mathcal P_\A \bfH^k(M)$$
are bounded. Moreover, $\mathcal S_\A$ and $\mathcal P_\A$ continuously depend on $g$ and $\A$.

\begin{lemma}\label{projection perturbation}
Given $(g,\A)\in C^1(M)$, there exists $\epsilon>0$ small such that for any $(\tilde g,\tilde \A)$ with $\|(\tilde g,\tilde \A)-(g,\A)\|_{C^1(M)}\leq \epsilon$
\begin{align*}
\|\mathcal S_{\tilde \A,\tilde g}-\mathcal S_{\A,g}\|_{\bfL \to \bfL}\leq C\epsilon,\quad \|\mathcal P_{\tilde \A,\tilde g}-\mathcal P_{\A,g}\|_{\bfL \to \bfL}\leq C\epsilon
\end{align*}
with $C>0$ a locally uniform constant depending on $g$ and $\A$ only. 
\end{lemma}
A proof in the absence of $\A$ can be found in \cite[Lemma 1]{SU05}, similar arguments will work for the case with $\A$.



Given $h=[\alpha,f]$ on $M$, we can extend it by zero onto $M_1$, still denoted by $h$. We want to compare $I_\A h$ and $\tilde I_\A h$. Notice that on $M$
$$X(\tilde W_\A W^{-1}_\A)=\tilde W_\A \A W^{-1}_\A-\tilde W_\A \A W^{-1}_\A=0,$$
i.e. $\tilde W_\A W^{-1}_\A$ is constant along geodesics on $M$.
Given $\gamma$ a geodesic on $M_1$ connection boundary points of $\p M_1$, assume that $\gamma\cap M\neq \emptyset$ 
\begin{align*}
\tilde I_\A h(\gamma)&=\int_{\gamma} \tilde W_\A(\gamma,\dgamma) h(\gamma,\dgamma)\,dt=\int_{\gamma\cap M} \tilde W_\A(\gamma,\dgamma) h(\gamma,\dgamma)\,dt\\
& =\int_{\gamma\cap M}\tilde W_\A(\gamma,\dgamma)W^{-1}_\A(\gamma,\dgamma)W_\A(\gamma,\dgamma)h(\gamma,\dgamma)\,dt\\
& =C(\gamma)\int_{\gamma\cap M} W_\A(\gamma,\dgamma)h(\gamma,\dgamma)\,dt\\
& =C(\gamma)I_\A h (\gamma),
\end{align*}
where $C(\gamma)$ is some constant that depends on $\gamma$ and is known if $g$ and $\A$ are given. Thus once $I_\A h$ is given, we know the values of $\tilde I_\A h$ and vice versa. From now on, we use $I_\A$ to represent both ray transforms.

Since $C^\infty(M_1)$ is dense in $\bfL(M_1)$, it is easy to check that $I_\A:\bfL(M_1)\to L^2_\mu(\p_+SM_1)$ is bounded, here $L^2_\mu(\p_+SM_1)$ is the $L^2$ space on $\p_+SM_1$ under the measure $d\mu=\<v,\nu\>_g\,d\Sigma^{2n-2}$ with $d\Sigma^{2n-2}$ the standard measure on $\p SM_1$. So the adjoint $I^*_\A: L^2_\mu(\p_+SM_1)\to \bfL(M_1)$ is bounded too. By a simple calculation, one can show that the integral expression of $I^*_\A$ is
$$I^*_\A u (x)=\begin{pmatrix} (\int_{S_xM_1} g_{i\ell}(x)v^\ell\tilde W^*_\A(x,v)u^\sharp(x,v)\,dv)\,dx^i \\[.5em] \int_{S_xM_1} \tilde W^*_\A(x,v)u^\sharp(x,v)\,dv \end{pmatrix},$$
where $W^*_\A$ is the conjugate transpose of $W_\A$, $u^\sharp$ is the invariant extension of $u$ along geodesics, i.e. $u^\sharp(x,v)=u(\gamma_{x,v}(\tau_-(x,v)),\dgamma_{x,v}(\tau_-(x,v)))$ with $\tau_-(x,v)$ the negative exit time of $\gamma_{x,v}$ from $\p M_1$.  

\begin{rem}\label{boundedness of I and I*}
One can also consider the boundedness of $I_\A$ and $I^*_\A$ on $H^k$ spaces for $k\geq 0$. Notice that we only consider $h=[\alpha,f]$ with $\supp h\subset M$, and $\tilde W_\A$ is smooth in $SM^o_1$ (so is $\tilde W^*_\A$), it is not difficult to check that 
\begin{align*}
I_\A: \bfH^k_c(M_1^o)\to H^k_c((\p_+SM_1)^o)\quad \mbox{and} \quad I^*_\A: H^k_c((\p_+SM_1)^o)\to \bfH^k(M_1)
\end{align*}
are bounded on simple $M_1$ for $k\geq 0$. See e.g. \cite[Theorem 4.2.1]{Sh94} and \cite[Proposition 5.3]{PZ15} for the ordinary geodesic ray transform.
\end{rem}


\section{Stability estimates}

We will study the microlocal properties of the operator $N_\A$ and prove Theorem \ref{linear stability} in this section.

\subsection{Ellipticity of $N_\A$}

\begin{lemma}\label{elliptic PsiDO}
$N_{\A}$ is a $\Psi$DO of order $-1$ in $M_1^o$. It is elliptic on solenoidal pairs at  any $(x,\xi)\in T^*M_1^o\setminus 0$.
\end{lemma}

\begin{proof}
Notice that $N_{\A}$ is an operator acting on pairs, similar cases was considered before in \cite{DPSU07, HS10}. It is not difficult to check that the integral operator has the following form
$$N_{\A}\begin{pmatrix} \alpha \\ f\end{pmatrix}=\begin{pmatrix} N_{\A}^{11} & N_{\A}^{10} \\[.5em] N_{\A}^{01} & N_{\A}^{00} \end{pmatrix}\begin{pmatrix}\alpha \\ f \end{pmatrix},$$
where
\begin{align*}
(N_{\A}^{11}\alpha)_i(x)& =\int_{S_xM_1}\int g_{i\ell}(x)v^\ell \tilde W_\A^*(x,v)\tilde W_\A(\gamma_{x,v}(t),\dot\gamma_{x,v}(t))\alpha_j(\gamma_{x,v}(t))\dot\gamma_{x,v}^j(t)\,dtdv,\\
(N_{\A}^{10}f)_i(x)& =\int_{S_xM_1}\int g_{i\ell}(x)v^\ell \tilde W_\A^*(x,v)\tilde W_\A(\gamma_{x,v}(t),\dot\gamma_{x,v}(t))f(\gamma_{x,v}(t))\,dtdv,\\
N_{\A}^{01}\alpha(x)& =\int_{S_xM_1}\int \tilde W_\A^*(x,v)\tilde W_\A(\gamma_{x,v}(t),\dot\gamma_{x,v}(t))\alpha_j(\gamma_{x,v}(t))\dot\gamma_{x,v}^j(t)\,dtdv,\\
N_{\A}^{00}f(x)& =\int_{S_xM_1}\int \tilde W_\A^*(x,v)\tilde W_\A(\gamma_{x,v}(t),\dot\gamma_{x,v}(t))f(\gamma_{x,v}(t))\,dtdv.
\end{align*}
Then following \cite[Proposition 4.1]{DPSU07}, it is easy to see that $N_{\A}$ is a $\Psi$DO of order $-1$ in $M_1^o$. Moreover the principal symbol $\sigma_p(N_{\A})$ satisfies
\begin{align*}
\sigma_p(N_{\A}^{11})_i^j(x,\xi)& =\int_{S_xM_1} g_{i\ell}(x)v^\ell \tilde W_\A^*(x,v)\tilde W_\A(x,v)v^j\delta(\xi\cdot v)\,dv,\\
\sigma_p(N_{\A}^{10})_i(x,\xi)& =\int_{S_xM_1} g_{i\ell}(x)v^\ell \tilde W_\A^*(x,v)\tilde W_\A(x,v)\delta(\xi\cdot v)\,dv,\\
\sigma_p(N_{\A}^{01})^j(x,\xi)& =\int_{S_xM_1} \tilde W_\A^*(x,v)\tilde W_\A(x,v)v^j\delta(\xi\cdot v)\,dv,\\
\sigma_p(N_{\A}^{00})(x,\xi)& =\int_{S_xM_1} \tilde W_\A^*(x,v)\tilde W_\A(x,v)\delta(\xi\cdot v)\,dv.
\end{align*}

Now for $(x,\xi)\in T^*M_1^o\setminus 0$ given $[\alpha,f]$ in the kernel of $\sigma_p(\delta_\A)(x,\xi)$, i.e. $g^{ij}\alpha_i\xi_j=0$. Notice that at a fixed point $x$, we can assume that the geometry is trivial, i.e. $g_{ij}(x)=\delta_{ij}$, so we can identify $\xi=(\xi_i)$ with its dual $\xi^*=(\xi^i)=(g^{ij}\xi_j)$. Assume that $(\sigma_p(N_{\A})[\alpha,f],[\alpha,f])=0$, which implies that
$$0=\int_{S_xM_1}\Big|\tilde W_\A(x,v)(\alpha_i(x)v^i+f)\Big|^2\delta(\xi\cdot v)\,dv.$$
Thus $\tilde W_\A(x,v)(\alpha_i v^i+f)=0$ for $(x,v)\in S_xM_1\cap \xi^{\perp}$. Thus we can find $n$ vectors $v_1,\cdots, v_n$ from $S_xM_1\cap \xi^\perp$ such that $\{v_2-v_1, \cdots, v_n-v_1\}$ form a basis of $\xi^\perp$. Since $\tilde W_\A(x,v)$ is invertible, we get $\alpha_iv_j^i+f=0$ for $j=1,\cdots,n$, thus
$$\alpha_i(v_j-v_1)^i=0,\quad j=2,\cdots,n.$$
The fact that $\{v_2-v_1,\cdots,v_n-v_1\}$ is a basis for $\xi^\perp$, together with the assumption $\xi^i\alpha_i=0$, implies that $\alpha=0$. Therefore $f=0$ too, and this proves the lemma.
\end{proof}

\begin{rem}\label{generalize 1}
Lemma \ref{elliptic PsiDO} still holds under the microlocal condition mentioned in Remark \ref{microlocal condition}. In fact the microlocal condition implies that there exists a smooth cut-off function $\rho$ on $SM_1$ such that for any $(x,\xi)\in T^*M_1^o$ there is $v\in S_x M_1$ so that $\rho(x,v)\neq 0$. In the mean time, $\tilde W_\A$ is well-defined and smooth in an open neighborhood of $\supp \rho$, thus one can show that $N'_\A=I^*_\A\,\rho^2\, I_\A$ is an elliptic $\Psi$DO of order $-1$ acting on solenoidal pairs.
\end{rem}

\subsection{Stability up to an error}

By Lemma \ref{elliptic PsiDO}, $(N_{\A},D \delta_\A)$ form an elliptic system in $M_1^o$, here $D=\begin{pmatrix} \Lambda \\ 0 \end{pmatrix}$
with $\Lambda$ a properly supported parametrix of $\Delta_\A$ in $M_1^o$ with principal symbol $|\xi|^{-2}$. Thus there is a parametrix for the system in $M_1^o$, denoted by $(P,Q)$, such that
\begin{equation}\label{parametrix 1}
PN_{\A}+QD \delta_\A=Id+K,
\end{equation}
where $P,\,Q$ are $\Psi$DO's of order $1$, $K$ is a smoothing operator.

Let $M'$ be a compact extension of $M$ such that $M\Subset M'\Subset M_1$, in particular we can choose $M'=M\cup (\p M\times [-\varepsilon/2,0))$ if one recalls the definition of $M_1$ in the introduction. Let $\chi$ be a smooth cut-off function on $M_1$ with $\supp \chi\subset M_1^o$ and $\chi=1$ in a neighborhood of $M'$. Given a pair $h=[\alpha,f]$ with $\supp h\subset M$, by Appendix B, there is a unique decomposition $h=h^s_{M_1}+d_\A \phi_{M_1}$ on $M_1$. Since $\supp \delta_\A \chi h^s_{M_1}\subset M_1\setminus M'$, by the pseudolocal property of $\Psi$DO's, $Q D\delta_\A \chi h^s_{M_1}$ is smooth near $M'$, then by \eqref{parametrix 1} we have
\begin{equation}\label{parametrix 2}
P N_{\A}\chi h^s_{M_1}=h^s_{M_1}+K_1 h^s_{M_1}=h^s_{M_1}+K'_1 h\quad \mbox{in}\quad (M')^o
\end{equation}
with $K_1$, $K'_1$ both smoothing operators, and $K'_1=K_1\mathcal S_{M_1}$.

For the term on the left-hand side of \eqref{parametrix 2}, notice that $\chi h=h$ on $M_1$
\begin{equation}\label{parametrix 3}
P N_{\A}\chi h^s_{M_1}=P N_{\A} h-P N_{\A}\chi d_\A \phi_{M_1},
\end{equation}
we want to rewrite the second term on the right-hand side as some compact operator acting on $h$. For this purpose, using the fact that $\Delta_\A\phi_{M_1}=\delta_\A h$, so $\phi_{M_1}=(\Delta_\A^D)^{-1}_{M_1}\delta_\A h$. In the mean time, let $\Lambda'$ be a parametrix for $\Delta_\A$ on $M_1^o$ such that $\Lambda'\Delta_\A=Id-\tilde K$ for some smoothing operator $\tilde K$, then \begin{align*}
((\Delta_\A^D)^{-1}_{M_1}-\Lambda')\delta_\A h& =(\Delta_\A^D)^{-1}_{M_1}\delta_\A h-\Lambda'\delta_\A\mathcal P_{M_1}h\\
&=(\Delta_\A^D)^{-1}_{M_1}\delta_\A h-\Lambda'\Delta_\A (\Delta_\A^D)^{-1}_{M_1}\delta_\A h\\
&=\tilde K (\Delta_\A^D)^{-1}_{M_1}\delta_\A h=\tilde K'h \quad \mbox{in} \quad (M')^o
\end{align*}
for some $\tilde K'$ with the same property as $\tilde K$. Thus in $(M')^o$
\begin{align*}
P N_{\A}\chi d_\A \phi_{M_1}&=P N_{\A}\chi d_\A (\Lambda'\delta_\A h+\tilde K'h)\\
&=P N_{\A}d_\A \chi(\Lambda'\delta_\A h+\tilde K'h)-P N_{\A}(d \chi) (\Lambda'\delta_\A h+\tilde K'h)\\
&=-P N_{\A}(d \chi) (\Lambda'\delta_\A +\tilde K')h=K_2 h,
\end{align*}
where $K_2$ is a $\Psi$DO of order $-1$. Now by \eqref{parametrix 2} and \eqref{parametrix 3}
\begin{equation}\label{parametrix 4}
P N_{\A}h=h^s_{M_1}+K_3 h\quad \mbox{in} \quad (M')^o
\end{equation}
where $K_3$ is a new $\Psi$DO of order $-1$. Thus
\begin{equation}\label{estimate 1}
\|h^s_{M_1}\|_{\bfL(M')}\lesssim \|N_{\A}h\|_{\bfH^1(M_1)}+\|h\|_{\bfH^{-1}(M_1)}.
\end{equation}

Next we want to change the term on the left-hand side of \eqref{estimate 1} from the $L^2$ norm of $h^s_{M_1}$ to the $L^2$ norm of $h^s_M$. Notice that 
\begin{equation}\label{change}
h^s_M=h-d_\A \phi_M=h^s_{M_1}+d_\A (\phi_{M_1}-\phi_M),
\end{equation}
denote $u=\phi_{M_1}-\phi_M$, then $u$ satisfies the elliptic boundary value problem
\begin{equation*}
\Delta_\A u=0,\quad u|_{\p M}=\phi_{M_1}|_{\p M},
\end{equation*}
and the following estimate holds
\begin{equation}\label{estimate 2}
\|d_\A(\phi_{M_1}-\phi_M)\|_{\bfL(M)}\lesssim \|\phi_{M_1}-\phi_M\|_{H^1(M)} \lesssim \|\phi_{M_1}\|_{H^{1/2}(\p M)}.
\end{equation}

By \eqref{parametrix 4} and the fact that $\supp h\subset M$,
\begin{equation}\label{parametrix 5}
d_\A \phi_{M_1}=-h^s_{M_1}=-PN_{\A}h+K_3h \quad \mbox{in} \quad M'\setminus M.
\end{equation}
For $\varepsilon$ small enough, $M'\setminus M=\p M\times [-\varepsilon/2,0)$ is with in some semigeodesic neighborhood of $\p M'$, so that for any $x=(x',t)\in M'\setminus M$, there is a unique geodesic $\gamma_x:[0,t]\to M'\setminus M$ normal to $\p M'$ with $\gamma_x(0)=(x',0)$ and $\gamma_x(t)=x$. Thus by the fundamental theorem of calculus and the fact $X(\tilde W_\A\phi_{M_1})=\tilde W_\A d_\A \phi_{M_1}$ on $SM_1$, we get
\begin{equation*}
\tilde W_\A(x,\p_t)\phi_{M_1}(x)=\tilde W_\A((x',0),\p_t)\phi_{M_1}(x',0)+\int_0^t \tilde W_\A(-PN_{\A}h+K_3h)((x',s),\p_t)\,ds.
\end{equation*}
On the other hand, $\phi_{M_1}=(\Delta_\A^D)^{-1}_{M_1}\delta_\A h$, we define $\phi_{M_1}(x',0)=K_4h(x)$. Since $h=0$ outside $M$, $\phi_{M_1}$ is smooth near $\p M'$, this implies that $K_4$ is a smoothing operator in $M'\setminus M$. Therefore
\begin{equation}\label{integral eq}
\phi_{M_1}(x)=\tilde U_\A(x,\p_t)\Big(\tilde W_\A (K_4 h)(x)+\int_{\gamma_x}\tilde W_\A(-PN_{\A}h+K_3 h)\,ds\Big).
\end{equation}
By \eqref{integral eq}, \eqref{parametrix 5} and the trace theorem
\begin{equation}\label{estimate 3}
\|\phi_{M_1}\|_{H^{1/2}(\p M)}\lesssim \|\phi_{M_1}\|_{H^1(M'\setminus M)}\lesssim \|N_{\A}h\|_{\bfH^1(M_1)}+\|h\|_{\bfH^{-1}(M_1)}.
\end{equation}
Combine \eqref{estimate 1}, \eqref{change}, \eqref{estimate 2} and \eqref{estimate 3} we achieve the following estimate

\begin{lemma}\label{stability up to error}
For any $h\in \bfL(M)$
\begin{equation}\label{estimate 4}
\|h^s_{M}\|_{\bfL(M)}\lesssim \|N_{\A}h\|_{\bfH^1(M_1)}+\|h\|_{\bfH^{-1}(M_1)}.
\end{equation}
\end{lemma}

\medskip

\begin{rem}
By \eqref{estimate 4}, if $h\in \mathcal S_\A\bfL(M)\cap \Ker I_{\mathcal A}$, then $N_{\mathcal A}h=0$, so
$$\|h\|_{\bfL(M)}\lesssim \|h\|_{\bfH^{-1}(M_1)}.$$
Since the inclusion $\bfL(M)\hookrightarrow \bfH^{-1}(M_1)$ is compact, it is easy to see that this implies that the space $\mathcal S_\A\bfL(M)\cap \Ker I_{\mathcal A}$ has finite dimension. Moreover, by \eqref{parametrix 1} and the pseudolocal property, $h$ is smooth in the interior of $M$. Indeed one can show that $\mathcal S_\A\bfL(M)\cap \Ker I_{\mathcal A}$ is included in $C^\infty(M)$ \cite{SU04,SU05}. This implies that the s-injectivity on $\bfL(M)$ is equivalent to the s-injectivity on $C^\infty(M)$.
\end{rem}

\subsection{Generic stability}

\begin{proof}[Proof of Theorem \ref{linear stability} (1)]
To prove part (1) of Theorem \ref{linear stability} we need the functional analysis lemma below, see \cite[Prop. V.3.1]{Ta81}.
\begin{lemma}\label{functional lemma}
Let $X,\,Y$ and $Z$ be Banach spaces, $T:X\to Y$ be an injective bounded linear operator, and $K:X\to Z$ be a compact operator. If for any $x\in X$
$$\|x\|_X\lesssim \|Tx\|_Y+\|Kx\|_Z,$$
then the following improved estimate holds
$$\|x\|_X\lesssim \|Tx\|_Y.$$
\end{lemma}
Now let $X=\mathcal S_\A\bfL(M)$, $Y=\bfH^1(M_1)$ and $Z=\bfH^{-1}(M_1)$, let $T=N_\A$ and $K$ be the inclusion map $\bfL(M)\hookrightarrow \bfH^{-1}(M_1)$. Assume that $I_{\A}$ is s-injective, then it is easy to see that $N_\A:\mathcal S_\A \bfL(M)\to \bfH^1(M_1)$ is injective. Notice that $N_\A \mathcal S_\A=N_\A$, by Lemma \ref{stability up to error} and Lemma \ref{functional lemma} we have the following stability estimate
\begin{equation}\label{stability estimate}
\|h^s_M\|_{\bfL(M)}\lesssim \|N_{\A}h\|_{\bfH^1(M_1)}
\end{equation}
for any $h\in \bfL(M)$. 
\end{proof}

Next we want to show that \eqref{stability estimate} is still true for $(\tilde g,\tilde \A)$ in some sufficiently small neighborhood of $(g,\A)$ under proper H\"older norms. We need the following lemma on the continuous dependence of $N_{g,\A}$ on $(g,\A)$.

\begin{lemma}\label{small perturbation}
Given $(g,\A)\in C^\infty(M_1)$, let $(\tilde g,\tilde \A)$ satisfy $\|g-\tilde g\|_{C^4(M_1)}\leq \epsilon$, $\|\A-\tilde \A\|_{C^3(M_1)}\leq \epsilon$ for some sufficiently small $\epsilon>0$, then the manifold $(M_1,\tilde g)$ is still simple and there exists a constant $C>0$ which only depends on $g,\A$ such that 
$$\|(N_{g,\A}-N_{\tilde g,\tilde \A})h\|_{\bfH^1(M_1)}\leq C\epsilon \|h\|_{\bfL(M)},$$
for any $h\in \bfL(M)$. 
\end{lemma}

The proof of Lemma \ref{small perturbation} can be carried out in the same spirit of \cite[Proposition 4]{FSU08} and \cite[Proposition 3]{HS10}, see the related references for more details. In particular, $\|X-\tilde X\|_{C^3}\leq \|g-\tilde g\|_{C^4}$ where $X$ and $\tilde X$ are the generating vector fields of the geodesic flows under metric $g$ and $\tilde g$ respectively. In the mean time, $\|\tilde W_{g,\A}-\tilde W_{\tilde g,\tilde A}\|_{C^3}\lesssim \epsilon$ under the assumptions of the lemma \cite[Lemma 5, 6]{Ho13}.

\begin{proof}[Proof of Theorem \ref{linear stability} (2)]
Given $h\in \bfL(M)$, by Lemma \ref{projection perturbation}, Lemma \ref{small perturbation} and \eqref{stability estimate}
\begin{align*}
\|h^s_{M,\tilde g,\tilde \A}\|_{\bfL(M)}\leq &\|\mathcal S_{g,\A}h^s_{M,\tilde g,\tilde \A}\|_{\bfL(M)}+\|(\mathcal S_{g,\A}-\mathcal S_{\tilde g,\tilde \A})h^s_{M,\tilde g,\tilde \A}\|_{\bfL(M)}\\
\leq& C_0\|N_{g,\A}h^s_{M,\tilde g,\tilde \A}\|_{\bfH^1(M_1)}+C\epsilon \|h^s_{M,\tilde g,\tilde \A}\|_{\bfL(M)}\\
 \leq &C_0\|N_{\tilde g,\tilde \A}h^s_{M,\tilde g,\tilde\A}\|_{\bfH^1(M_1)}+C_0\|(N_{g,\A}-N_{\tilde g,\tilde \A})h^s_{M,\tilde g,\tilde\A}\|_{\bfH^1(M_1)}\\
 & \quad \quad +C\epsilon \|h^s_{M,\tilde g,\tilde \A}\|_{\bfL(M)}\\
\leq & C_0\|N_{\tilde g,\tilde \A}h\|_{\bfH^1(M_1)}+C_0C\epsilon\|h^s_{M,\tilde g,\tilde \A}\|_{\bfL(M)}+C\epsilon\|h^s_{M,\tilde g,\tilde\A}\|_{\bfL(M)}.
\end{align*}
Notice that $h=h^s_{M,\tilde g,\tilde\A}+d_{\tilde\A}\phi_{M,\tilde g,\tilde \A}$ on $M$, $\phi_{M,\tilde g,\tilde\A}|_{\p M}=0$. We may extend $h$, $h^s_{M,\tilde g,\tilde\A}$ and $\phi_{M,\tilde g,\tilde\A}$ by zero onto $M_1$ (so $\phi_{M,\tilde g,\tilde\A}\in H_0^1(M_1)$), then $N_{\tilde g,\tilde A}d_{\tilde\A}\phi_{M,\tilde g,\tilde\A}=0$ also as $\phi_{M,\tilde g,\tilde \A}|_{\p M_1}=0$ too, i.e. $N_{\tilde g,\tilde\A}h^s_{M,\tilde g,\tilde\A}=N_{\tilde g,\tilde\A}h$.

Now let $\epsilon\leq \frac{1}{2(C_0C+C)}$, we get that 
$$\|h^s_{M,\tilde g,\tilde \A}\|_{\bfL(M)}\leq 2C_0\|N_{\tilde g,\tilde \A}h\|_{\bfH^1(M_1)}.$$
Notice that $C_0$ is the constant from Theorem \ref{linear stability} (1), which only depends on $g,\,\A$, this completes the proof.
\end{proof}


\section{S-injectivity in the real-analytic category}

In what follows, analytic means real-analytic. We will first show that if $g$ and $\A$ are analytic and $I_\A h=0$, then $h^s$ is analytic on $M$, i.e. $h^s\in \mathbb A(M)$. It is easy to see that if $\A$ is analytic, then $W_\A$ is analytic too. We denote $\WF_a(h)$ the analytic wave front of $h$.  

\begin{prop}\label{interior analytic}
Assume that $g$ and $\A$ are analytic. Given $(x_0,\xi_0)\in T^*M^o\setminus 0$, let $\gamma_0$ be a geodesic through $x_0$ normal to $\xi_0$. If for some $h=[\alpha,f]\in \bfL(M)$, $I_\A h(\gamma)=0$ for $\gamma$ in a neighborhood of $\gamma_0$, and $\delta_\A h=0$ near $x_0$, then 
$$(x_0,\xi_0)\not\in \WF_a(h).$$
\end{prop}

Since $\WF_a(h)$ is closed, above proposition also implies that a neighborhood of $(x_0,\xi_0)$ is away from $\WF_a(h)$.

\begin{proof}
Assume $\gamma_0:[\ell^-,\ell^+]\to M_1$, $\gamma_0(0)=x_0$ and $\ell^-<0,\,\ell^+>0$ with $\gamma_0(\ell^-)$, $\gamma_0(\ell^+)\in M_1\setminus M$. We can define analytic coordinates in a tubular neighborhood $U$ of $\gamma_0$ in $M_1$ with $x=(x',t),\, x'=(x^1,\cdots,x^{n-1})$ such that $U=\{|x'|<\varepsilon,\,\ell^--\varepsilon<t<\ell^++\varepsilon\}$ for some small $\varepsilon>0$, $x_0=0$ and $\gamma_0=\{(0,\cdots,0,t):\ell^-\leq t\leq \ell^+\}$. Then if $\varepsilon$ is small enough, $(x',\ell^-),\,(x',\ell^+)\in M_1\setminus M$ for any $|x'|<\varepsilon$. See also \cite[Sec. 2.1]{SU08}. In particular, one can assume that the geometry at $x_0=0$ is trivial, i.e. $g_{ij}(0)=\delta_{ij}$, then $\xi_0=(\xi'_0,0)$. Thus $\gamma_0=\gamma_{x_0,v_0}$ with $v_0=(0,\cdots,0,1)$.

For curves in a neighborhood of $\gamma_0$, we give a local parameterization under the analytic coordinates above. Given $|z'|<\varepsilon$ and $|v'|<1$ small, we consider curves $\gamma_{z',v'}:=\gamma_{(z',0),(v',1)}$ with $\gamma_{z',v'}(t)=\exp_{(z',0)}(t(v',1))$. Then for $|z'|<2\varepsilon/3$ and $|v'|\ll 1$, the curve $\gamma_{z',v'}([\ell^-,\ell^+])$ will stay in $U$ as well and $\gamma_{z',v'}(\ell^\pm)\in M_1\setminus M$. Thus $I_\A h(\gamma_{z',v'})=0$ for $|z'|<2\varepsilon/3,\,|v'|\ll 1$. 


Much of the complexity of analytic microlocal calculus is due to the difficulty of
localizing in the analytic category, as there are no suitable cut-off functions. Similar to \cite{SU08}, we instead use a sequence of cut-off functions $\chi_N\in C^\infty_c(\R^{n-1})$ satisfying $\supp (\chi_N)\subset \{|z'|<2\varepsilon/3\},\, \chi_N=1$ for $|z'|<\varepsilon/3$ and 
\begin{equation}\label{cut-off}
|\p^\beta \chi_N(z')|<(CN)^{|\beta|},\,\forall z'\in \R^{n-1},\,|\beta|<N
\end{equation}
for some $C>0$ independent of $N$. The existence of such cut-off functions can be found in e.g. \cite{Tr80}.

Let $\mu\gg 1$ be a large parameter, then for $\xi=(\xi',\xi^n)$ in a complex neighborhood of $\xi_0$ we have
$$\int e^{i\mu z'\cdot \xi'}\chi_N(z')\int \tilde W_\A(\gamma_{z',v'}(t),\dot\gamma_{z',v'}(t))\Big(\alpha(\dot\gamma_{z',v'}(t))+f(\gamma_{z',v'}(t))\Big)\,dt dz'=0.$$
Notice that with the help of the cut-off function $\chi_N$, we may make analytic coordinates change $(z',v',t)\to (x,\xi)$ near $v'=0$ with $v'=v'(\xi)$ (so $v'$ is sufficiently small when $\xi$ is close enough to $\xi_0$) such that $(v'(\xi),1)\cdot\xi=0$, $v'(\xi_0)=0$. In particular $x=\gamma_{z',v'}(t)$. Thus
\begin{equation}\label{int1}
\int e^{i\mu \varphi(x,\xi)} W_N(x,\xi)\Big(\alpha_i(x)u^i(x,\xi)+f(x)\Big)\,dx=0.
\end{equation}
Here $\varphi(x,\xi):=z'(x,v'(\xi))\cdot \xi'$ is the phase function. $W_N$ is an analytic matrix function for $\xi$ sufficiently close to $\xi_0$, independent of $N$ near $\gamma_0$ and satisfies \eqref{cut-off} too. $u(x,\xi)=(u^1(x,\xi),\cdots,u^n(x,\xi))$ is an analytic vector field. Note that $W_N(0,\xi)=\tilde W_\A(0,(v'(\xi),1))=\tilde W_\A(0,u(0,\xi))$. 

Now we are going to apply the method of complex stationary phase \cite{Sj82}, see also \cite{KSU07, SU08}. Notice that our phase function $\varphi$ is the same as the one considered in \cite{SU08}, in particular we have the following lemma.

\begin{lemma}\label{varphi}
The phase function $\varphi$ in \eqref{int1} satisfies the following properties:
\begin{enumerate}
\item $\p_{\xi}\p_x\varphi(0,\xi)=Id$, thus $\varphi$ is a non-degenerate phase function near $(0,\xi_0)$;
\item there exists $\delta>0$ such that if $\p_{\xi}\varphi(x,\xi)=\p_{\xi}\varphi(y,\xi)$ for some $x\in U$, $|y|<\delta$ and $|\xi-\xi_0|<\delta$, then $x=y$.
\end{enumerate}
\end{lemma}

Now let $|y|<\delta$, $|\eta-\xi_0|<\delta/2$, let $\rho$ be a smooth cut-off function such that $\supp \rho\subset \{|\xi|<\delta\}$ and $\rho(\xi)=1$ for $|\xi|<\delta/2$. We multiple \eqref{int1} by 
$$\rho(\xi-\eta) e^{i\mu\Big(\frac{i}{2}(\xi-\eta)^2-\varphi(y,\xi)\Big)},$$
and integrate in $\xi$ to get
\begin{equation}\label{int2}
\int\int e^{i\mu \Phi(x,y,\xi,\eta)} \widetilde{W}_N(x,\xi,\eta)\Big(\alpha_i(x)u^i(x,\xi)+f(x)\Big)\,dx d\xi=0,
\end{equation}
where 
$$\Phi(x,y,\xi,\eta)=\frac{i}{2}(\xi-\eta)^2+\varphi(x,\xi)-\varphi(y,\xi),\quad \widetilde W_N(x,\xi,\eta)=\rho(\xi-\eta)W_N(x,\xi).$$

To estimate the left-hand side of \eqref{int2}, we first study the critical points of the function $\xi \mapsto \Phi(x, y, \xi, \eta)$. Note that
\[
 \p_\xi \Phi(x, y, \xi, \eta) = i(\xi-\eta) + \partial_\xi\varphi(x, \xi) - \partial_\xi\varphi(y, \xi).
\]
By Lemma \ref{varphi} (2), when $x=y$, the only critical point of $\Phi$ is $\xi_c=\eta$ which is non-degenerate. Therefore for $|x-y|\leq\delta/C_0$, some $C_0>0$, there is at most one (complex) critical point $\xi_c=\xi_c(x,y,\eta)$ in $|\xi-\eta|<\delta$, while none if $|x-y|>\delta/C_0$. 

Denote $\zeta:=\p_x\varphi(y,\eta)$, by Lemma \ref{varphi} (1) we can change variables $(x,y,\eta)\to (x,y,\zeta)$ in a sufficiently small neighborhood of $(0,0,\xi_0)$ and define 
$$\Psi(x,y,\zeta):=\Phi(x,y,\xi_c,\eta).$$
Then
$$\Psi(x,x,\zeta)=0,\quad \p_x\Psi(x,x,\zeta)=\zeta,\quad \p_y\Psi(x,x,\zeta)=-\zeta$$
and
$$\mbox{Im}\, \Psi(x,y,\zeta)>|x-y|^2/C.$$
Note that $\widetilde W_N$ is analytic and independent of $N$ on $|x-y|\leq \delta/C_0$, we apply the complex stationary phase lemma \cite[Theorem 2.8, 2.10]{Sj82} to \eqref{int2} to get that 
\begin{equation}\label{int3}
\int_{|x-y|\leq C} e^{i\mu \Psi(x,y,\zeta)} \widetilde{W}(x,y,\zeta;\mu)\Big(\alpha_i(x)u^i(x,y,\zeta)+f(x)\Big)\,dx=\mathcal O (e^{-\mu/C}),
\end{equation}
where $\widetilde W$ is an analytic matrix weight. Note that the right-hand side of \eqref{int3} is independent of $\delta$, for fixed $\delta$, we simply replace $\delta/C_0$ with some positive number $C$. 

For the following argument, we consider the left-hand side of \eqref{int3} as an operator with a matrix-valued symbol acting on the pair $[\alpha,f]$, thus \eqref{int3} can be rewritten as
\begin{equation}\label{int4}
\int_{|x-y|\leq C} e^{i\mu \Psi(x,y,\zeta)} P(x,y,\zeta;\mu)\begin{pmatrix}\alpha \\ f\end{pmatrix}\,dx=\mathcal O (e^{-\mu/C}),
\end{equation}
where $P(x,y,\zeta;\mu)=\widetilde W(x,y,\zeta;\mu)\begin{pmatrix} u(x,y,\zeta)Id_{k\times k} & Id_{k\times k} \end{pmatrix}$ is a classical analytic symbol. Notice that $u(0,\xi_0)=v_0=(0,\cdots,0,1)$, the principal symbol satisfies
$$\sigma_p(P)(0,0,\xi_0)=\tilde W_\A(0,u(0,\xi_0))\begin{pmatrix} u(0,\xi_0)Id_{k\times k} & Id_{k\times k} \end{pmatrix}=\tilde W_\A(0,v_0)\begin{pmatrix} v_0Id_{k\times k} & Id_{k\times k} \end{pmatrix}$$
(note that $\alpha$ and $f$ are vector-valued 1-forms and functions). Recall that $v_0\perp \xi_0$, in any small neighborhood of $v_0$, we can find another $n-1$ unit vectors $v_1,\cdots, v_{n-1}$ at $x_0=0$ such that $v_j\perp \xi_0$, $j=1,\cdots,n-1$ and $\{v_1-v_0,\cdots,v_{n-1}-v_0\}$ form a basis of the orthogonal plane $\xi_0^{\perp}$. Moreover, the geodesic $\gamma_{x_0,v_j}$, $j=1,\cdots,n-1$ will stay in $U$ for the time interval $[\ell^-,\ell^+]$ with $\gamma_{x_0,v_j}(\ell^\pm)\in M_1\setminus M$ if $v_j$ is sufficiently close to $v_0$. We repeat the argument above under the new coordinates change $(z',v,t)\to (x,\xi)$ with $v(\xi_0)=v_j$ to get totally $n$ equations
\begin{equation}\label{int5}
\int_{|x-y|\leq C} e^{i\mu \Psi_j(x,y,\zeta)} P_j(x,y,\zeta;\mu)\begin{pmatrix}\alpha \\ f\end{pmatrix}\,dx=\mathcal O (e^{-\mu/C}),\quad j=0,1,\cdots,n-1
\end{equation}
with $\Psi_0$ and $P_0$ being exactly $\Psi$ and $P$ in \eqref{int4},
\begin{equation}\label{symbol 1}
\sigma_p(P_j)(0,0,\xi_0)=\tilde W_\A(0,v_j)\begin{pmatrix} v_j Id_{k\times k} & Id_{k\times k} \end{pmatrix}.
\end{equation}

To make the system \eqref{int5} into an elliptic system, we need one more equation from the property that $\delta_\A[\alpha,f]=0$ in some neighborhood $V$ of $x_0=0$. As in \cite{SU08}, let $\chi$ be a smooth cut-off function on $M_1$ supported in $V\cap U$ and $\chi=1$ near $0$, thus $\chi\delta_\A[\alpha,f]\equiv 0$. Applying the integration by parts,
\begin{align*}
0 & =\int\frac{1}{\mu}e^{i\mu\Psi_0(x,y,\zeta)}\chi\delta_\A[\alpha,f]\,dx=-\int\frac{1}{\mu}d_\A(e^{i\mu\Psi_0(x,y,\zeta)}\chi)\begin{pmatrix}\alpha \\ f \end{pmatrix}\,dx\\
& =-\int e^{i\mu\Psi_0(x,y,\zeta)}\begin{pmatrix} i\chi\p_x\Psi_0(x,y,\zeta)+\frac{1}{\mu}(d\chi+A\chi) & \frac{1}{\mu}\Phi\chi \end{pmatrix}\begin{pmatrix}\alpha \\ f \end{pmatrix}\,dx\\
& =-\int e^{i\mu\Psi_0(x,y,\zeta)}Q(x,y,\zeta)\begin{pmatrix}\alpha \\ f \end{pmatrix}\,dx.
\end{align*}
Notice that $\p_x\Psi_0(x,x,\zeta)=\zeta$, therefore 
\begin{equation}\label{symbol 2}
\sigma_p(Q)(0,0,\xi_0)=\begin{pmatrix}\xi_0 Id_{k\times k} & 0\end{pmatrix}.
\end{equation}

We combine the $n+1$ equations above into one system
\begin{align*}
& \int_{|x-y|\leq C} \diag \{e^{i\mu\Psi_0},\cdots,e^{i\mu\Psi_{n-1}},e^{i\mu\Psi_0}\}(x,y,\zeta)\begin{pmatrix} P_0(x,y,\zeta;\mu) \\ \vdots \\ P_{n-1}(x,y,\zeta;\mu) \\ Q(x,y,\zeta) \end{pmatrix}\begin{pmatrix}\alpha \\ f\end{pmatrix}\,dx\\
=& \int_{|x-y|\leq C} \diag \{e^{i\mu\Psi_0},\cdots,e^{i\mu\Psi_{n-1}},e^{i\mu\Psi_0}\}(x,y,\zeta){\bf P}(x,y,\zeta;\mu)\begin{pmatrix}\alpha \\ f\end{pmatrix}\,dx=\mathcal O(e^{-\mu/C}),
\end{align*}
where ${\bf P}$ is a matrix-valued classical analytic symbol near $x=0$. We claim that the system is elliptic at $(0,0,\xi_0)$, which is equivalent to the invertibility of the principal symbol $\sigma_p({\bf P})$ at $(0,0,\xi_0)$. Assume $\sigma_p({\bf P})[\alpha,f](0,0,\xi_0)=0$, by \eqref{symbol 1} and \eqref{symbol 2}, 
\begin{equation*}
\tilde W_\A(0,v_j)(v_j\cdot\alpha+f)=0,\,j=0,1,\cdots,n-1;\quad \xi_0\cdot \alpha=0.
\end{equation*}
By an argument almost identical to the one in Lemma \ref{elliptic PsiDO} one can show that $\alpha=0$ and $f=0$, so the system with classical analytic symbol ${\bf P}$ is elliptic at $(0,0,\xi_0)$.

Now follow the argument in \cite{SU08, HS10}, see also \cite[Prop. 6.2]{Sj82}, we can reduce the system to a new one of the form
$$\int_{|x-y|\leq C}e^{i\mu\Psi_0(x,y,\zeta)}{\bf Id}\begin{pmatrix} \alpha \\ f \end{pmatrix}\,dx=\mathcal O(e^{-\mu/C}),$$
for $(y,\zeta)$ close to $(0,\xi_0)$. This shows that $(0,\xi_0)\not\in \WF_a([\alpha,f])$ in the sense of \cite[Def. 6.1]{Sj82}, and completes the proof.
\end{proof}

\begin{rem}\label{generalize 2}
One can easily see from the proof of above proposition that for each $(x,\xi)\in T^*M^o\setminus 0$, we only require the existence of some $v\in S_xM$ conormal to $\xi$ and an arbitrarily small neighborhood of $\gamma_{x,v}$ whose elements are all `good' geodesics, in particular this is true under the microlocal condition of Remark \ref{microlocal condition}.
\end{rem}

Proposition \ref{interior analytic} shows that $I_\A h\equiv 0$ implies that $h^s$ is analytic in $M^o$, i.e. $h^s\in \mathbb A(M^o)$, the next lemma shows that $h^s$ indeed is analytic upto the boundary $\p M$. The proof of the lemma is almost identical to the one of \cite[Lemma 6]{SU08}, so we omit it here.

\begin{lemma}\label{analytic up to boundary}
Assume that $(M,g)$ is simple, $g$ and $\A$ are analytic, if $I_\A h\equiv 0$, then $h^s\in \mathbb A(M)$.
\end{lemma}



We also need the following lemma which will be useful for the proof of Theorem \ref{analytic s-injective}.

\begin{lemma}\label{infinite vanishing on boundary}
Assume that $I_{\A}h=0$ for some $h\in C^\infty(M)$, then there exists $\phi\in C^\infty(M)$ with $\phi|_{\p M}=0$ such that if $\tilde h:=h-d_\A \phi=[\tilde \alpha,\tilde f]$, then
$$\p^{\sigma}\tilde h|_{\p M}=0$$
for all multiindices $\sigma$ and 
$$\tilde \alpha_n=0$$
 in boundary normal coordinates near $\p M$. 
\end{lemma}

\begin{proof}
If $h=[\alpha,f]$, we consider the following equation in boundary normal coordinates $(x',x^n),\,0\leq x^n<\varepsilon\ll 1$ near $\p M=\{x^n=0\}$
\begin{equation}\label{solve for phi}
\p_n\phi+A_n\phi=\alpha_n,\quad \phi|_{x^n=0}=0.
\end{equation}
We solve \eqref{solve for phi} by integrating along $x^n$ to get $\phi$ in a neighborhood of $\p M$, and one can check that $\phi$ is well-defined (independent of local coordinates) and smooth near $\p M$. By multiplying $\phi$ with a proper cut-off function we can assume that $\phi$ is globally defined on $M$. Then we define $\tilde h=[\tilde \alpha,\tilde f]=h-d_\A\phi$, by \eqref{solve for phi} $\tilde \alpha_n=0$ near $x^n=0$.  

Next we show that $\tilde h$ vanishes to infinite order on $\p M$. Extend $\tilde h$ by zero to $M_1$, still denoted by $\tilde h$, then by the assumption $I_{\A} h=0$ we have $N_{\A}\tilde h=0$ in $M_1^o$, in particular $N_{\A}\tilde h$ is smooth near $M$. On the other hand, if we replace the condition $\alpha_i \xi^i=0$ by $\alpha_n=0$ in the proof of Lemma \ref{elliptic PsiDO}, it is not difficult to check that $(\sigma_p(N_{\A})(x,\xi)[\alpha,f],[\alpha,f])=0$ together with $\alpha_n=0$ implies that $[\alpha,f]=0$ for $\xi_n\neq 0$. This means that $N_{\A}$ is elliptic for $(x,\xi)$ conormal to $\p M$. Thus $N^*(\p M)\cap WF(\tilde h)=\emptyset$. Since $\tilde h=0$ in $M_1\setminus M$, we get that $\p_n^k \tilde h|_{x^n=0}=0$, $\forall k\geq 0$. After rewriting the conclusion in an invariant way, the proof is done.
\end{proof}

\begin{proof}[Proof of Theorem \ref{analytic s-injective}]
Assume $I_\A h=0$ and $h=h^s$, then by Lemma \ref{analytic up to boundary} $h\in \mathbb A(M)$. Applying Lemma \ref{infinite vanishing on boundary}, there exists $\phi\in C^\infty(M)$ with $\phi|_{\p M}=0$ such that $\tilde h=h-d_\A\phi$ satisfies $\p^\sigma\tilde h|_{\p M}=0$ for all multiindices $\sigma$. Since $\A$ and $h$ are analytic on $M$, and $\phi$ solves the equation \eqref{solve for phi}, then $\phi$, and therefore, $\tilde h$ are analytic near $\p M$ in the boundary normal coordinates. Since $\tilde h$ vanishes to infinite order on $x^n=0$, we get that $\tilde h$ actually vanishes in a neighborhood of $\p M$ in $M$, i.e. $h=d_\A \phi$ near $\p M$.

Now applying the analytic continuation argument as in the proof of \cite[Theorem 1]{SU08}, one can show that indeed $h=d_\A \phi_0$ in $M$ for some $\phi_0\in \mathbb A(M)$ with $\phi_0|_{\p M}=0$. However, $h=h^s$, this only can happen if $h=0$, so $I_\A$ is s-injective.
\end{proof}



\section{The X-ray transform of functions with matrix weights}\label{function case}

By modifying the two sections above, we can prove similar results for weighted geodesic ray transform acting on vector-valued functions, if the smooth weight (matrix) $W$ is invertible. Note that in the scalar case, this just means that the weight is non-vanishing \cite{FSU08}. The study of the local invertibility of such ray transforms was carried out in \cite{SUV16,PSUZ16}. Given $f\in C^\infty(M,\mathbb C^k)$ we define
$$I_W f(\gamma)=\int W(\gamma(t),\dot\gamma(t)) f(\gamma(t))\,dt.$$
In this case one can expect the kernel of $I_W$ to be empty, in particular this is true for real-analytic simple metric and weights.

\begin{thm}
Let $M$ be a real-analytic simple manifold with real-analytic metric $g$, let $W$ be real-analytic, then $I_W$ is injective.
\end{thm}

Define $N_W=I^*_W I_W$.

\begin{lemma}
$N_W$ is an elliptic $\Psi$DO of order $-1$ in $M^o_1$.
\end{lemma}

One can also show the following generic stability result.

\begin{thm}
Let $(M,g)$ be a simple manifold and $W$ be a smooth invertible weight on $SM_1$, assume that $I_W$ is injective,\\
(1) given $f\in L^2(M)$ the following stability estimate for $N_W$ holds
$$\|f\|_{L^2(M)}\leq C \|N_W f\|_{H^1(M_1)};$$
(2) there exists $0<\epsilon\ll 1$ such that the estimate in (1) remains true if $g$ and $W$ are replaced by $\tilde g$ and $\tilde W$ satisfying $\|\tilde g-g\|_{C^4(M_1)}\leq \epsilon$, $\|\tilde W-W\|_{C^3(M_1)}\leq \epsilon$. The constant $C>0$ can be chosen locally uniformly, only depending on $g,\,W$.
\end{thm}

Similar to \cite{FSU08}, above results hold on a general family of smooth curves under some microlocal condition associated with these curves. 

\section{The non-linear problem}\label{nonlinear problem}

Now we move to the non-linear problem of recovering the connection and Higgs field from the corresponding scattering data.

Given a geodesic $\gamma: [0,T]\to M$, let $\phi(t)=(\gamma(t),\dot{\gamma}(t))$ be the corresponding geodesic flow on $SM$. Define the matrix-valued function
\[
F(t)=W_{\B}(\phi(t))W_{\A}^{-1}(\phi(t)),
\]
by the fundamental theorem of calculus and the definitions of $W_\A,\,W_\B$ 
\begin{equation}\label{integral 1}
F(T)-F(0)=\int_0^T W_{\B}(\phi)(\B (\phi)-\A (\phi))W_{\A}^{-1}(\phi)\,dt.
\end{equation}
We define $\hat W$ by 
\[
\hat W U=W_{\B}UW_{\A}^{-1}, \qquad U\in C^{\infty}(SM;\mathbb C^{k\times k}),
\]
then the right-hand side of \eqref{integral 1} indeed gives the following weighted geodesic ray transform of $\B-\A$:
\begin{equation}\label{integral 2}
\int_{\gamma} \hat W(\B-\A)\,dt.
\end{equation}
By the definition of $\hat W$, it is obvious that $\hat W|_{\p_+SM}={\rm id}$. Given a matrix-valued function $U$ we have 
\begin{equation*}
X(W_{\B}UW_{\A}^{-1})=W_{\B}\B UW_{\A}^{-1}+W_{\B}(XU)W_{\A}^{-1}-W_{\B}U \A W_{\A}^{-1},
\end{equation*}
which implies that
\[
(X\hat W)U=\hat W(\B U-U \A).
\]
If we define $\hat{\A}$ by $\hat{\A}U=\B U-U\A$, we get exactly $X\hat W=\hat W\hat {\A}$, i.e. \eqref{integral 2} is the attenuated geodesic ray transform with the attenuation $\hat {\A}$, we denote it by $I_{\hat\A}$.

Similar to the linear problem, recall the discussion in Section 2, one can extend $\A$ and $\B$ onto $M_1$ in a stable way and consider the equivalent ray transform on $M_1$, still denoted by $I_{\hat \A}$. This works if we consider the integrand $\B-\A$ as extended by zero to $M_1$ (we treat the integrand $\B-\A$ and the weight $\hat W$ as independent with each other).

\begin{proof}[Proof of Theorem \ref{nonlinear thm}] 
We first consider the injectivity, if $C_\A=C_\B$, then $W_{\A}(\phi(T))=W_{\B}(\phi(T))$, i.e.\ $F(T)=F(0)={\rm id}$, so we get
\[
I_{\hat\A}(\B-\A)=0.
\]
Let $\hat\A_0 U=\B_0 U-U \A_0$, then since $\A_0,\,\B_0$ and $g_0$ are analytic (also extended to $M_1$ analytically), by Theorem \ref{analytic s-injective} $I_{\hat\A_0}$ is s-injective. By the assumptions of Theorem \ref{nonlinear thm} $\|\hat\A_0-\hat\A\|_{C^3}\lesssim \epsilon$.
Then Theorem \ref{linear stability} (2) implies that $I_{\hat\A}$ is s-injective, i.e. $\B-\A=d_{\hat\A}U$ for some $U\in C^\infty(M;\mathbb C^{k\times k})$, $U|_{\p M}=0$ w.r.t. $(g,\hat\A)$. Notice that $d_{\hat\A}U=dU+\B U-U \A$, thus $u=\id-U$ satisfies the transport equation
$$Xu+\hat \A u=0,\quad u|_{\p M}=\id,$$
in particular, $u$ is invertible. 
Let $p=u^{-1}$, we get that
$$\B=p^{-1}d_\A p,\quad p|_{\p M}=\id.$$
Equivalently, $B=p^{-1}dp+p^{-1}A p$ and $\Psi=p^{-1}\Phi p$.

\medskip

Then we consider the stability of the nonlinear problem. By the assumptions of Theorem \ref{nonlinear thm} $\iota^*\A=\iota^*\B$ and $\|\B-\A\|_{C^0(M)}\leq 2\epsilon$, it is not difficult to see that one can find $p:M\to GL(k,\mathbb C)$ such that $p|_{\p M}=\id$, $(dp,0)|_{\p M}=\B-\A|_{\p M}$ and $\|p-\id\|_{C^4(M)}\lesssim \epsilon$. Now consider
$$\A'=p^{-1}dp+p^{-1}\A p,$$ 
so $C_{\A'}=C_\A$. Moreover, $\B-\A'|_{\p M}=\B-\A-(dp,0)|_{\p M}=0$, we get that $\B-\A'\in H^1_c(M_1^o)$ if we extend $\B-\A'$ by zero on to $M_1$. Then it is easy to check that $\|\A'-\A_0\|_{C^3(M)}\leq C\epsilon$ for some $C>0$ depending on $\|\A_0\|_{C^3(M)}$. If we define $\hat\A'$ by $\hat\A' U=\B U-U\A'$, we have the similar bound that $\|\hat\A'-\hat\A_0\|_{C^3}\lesssim \epsilon$, therefore we can apply Theorem \ref{linear stability} to $\hat\A'$.

By Theorem \ref{decomposition thm}, there exists $U$ with $U|_{\p M}=0$ such that $\B-\A'=(\B-\A')^s_{M,g,\hat\A'}+d_{\hat\A'}U$, let $\id-U=V$, $V|_{\p M}=\id$, thus 
$$(\B-\A')^s_{M,g,\hat\A'}=dV+\B V-V\A'=d_{\hat\A'}V.$$
Moreover, the following estimate in H\"older norm holds
$$\|U\|_{C^2(M)}\lesssim \|\B-\A'\|_{C^2(M)}.$$
Since $\|\B-\A'\|_{C^2(M)}\lesssim \epsilon$ and $U|_{\p M}=0$, this implies that $V=\id-U$ is invertible. On the other hand, by Theorem \ref{linear stability} (2) and the assumptions of the theorem again 
\begin{align*}
\|(\B-\A')^s_{M,g,\hat\A'}\|_{\bfL(M)}\lesssim\|N_{g,\hat\A'}(\B-\A')\|_{\bfH^1(M_1)}.
\end{align*}
Since $\B-\A'\in H^1_c(M_1^o)$ (notice that $\B-\A'=0$ in $M_1\setminus M$), by Remark \ref{boundedness of I and I*} and \eqref{integral 1}
\begin{align*}
\|N_{g,\hat\A'}(\B-\A')\|_{\bfH^1(M_1)}&=\|\tilde I^*_{g,\hat\A'}\tilde I_{g,\hat\A'}(\B-\A')\|_{\bfH^1(M_1)}\\
&\lesssim \|\tilde I_{g,\hat\A'}(\B-\A')\|_{H^1(\p_+SM_1)}\\
&\lesssim \|I_{g,\hat\A'}(\B-\A')\|_{H^1(\p_+SM)}\\
&\lesssim \|(W_\B-W_{\A'})W^{-1}_{\A'}\|_{H^1(\p_-SM)}\\
&\lesssim \|C_\B-C_{\A'}\|_{H^1(\p_+SM)}=\|C_\B-C_\A\|_{H^1(\p_+SM)}.
\end{align*}
Notice that $W_{\A'}$ is smooth on $\p_-SM$. Here we use the equivalence of $\|\tilde I_{\hat\A'}h\|_{H^1(\p_+SM_1)}$ and $\|I_{\hat\A'}h\|_{H^1(\p_+SM)}$ for $h\in \bfH_0^1(M)$ due to the equivalence of $I_{\hat\A'}$ and $\tilde I_{\hat\A'}$ mentioned in Section 2, see e.g. \cite[Lemma 6.2]{PZ15}. 

Finally by the definition of $\A'$
\begin{align*}
dV+\B V-V \A'&= dV+\B V-V(p^{-1}dp+p^{-1}\A p)\\
& =((dV) p^{-1}+\B V p^{-1}-V p^{-1} (dp) p^{-1}-V p^{-1}\A)\, p\\
& =(d u+\B u-u \A)\, p=(\B-u d_\A u^{-1})V,
\end{align*}
where $u:=V p^{-1}$ is invertible, $u|_{\p M}=\id$.
Combine above results
$$\|\B-ud_\A u^{-1}\|_{\bfL(M)}\leq C'\|(\B-\A')^s_{M,g,\hat\A'}\|_{\bfL(M)}\leq C\|C_\B-C_\A\|_{H^1(\p_+SM)}.$$
Moreover, it is not difficult to check that the constant $C>0$ can be chosen locally uniformly near $g_0$ and $\A_0$.
\end{proof}

\begin{rem}\label{generalize 3}
When the manifold only satisfies the microlocal condition mentioned in Remark \ref{microlocal condition}, so is not simple in general, the scattering data is not well-defined on the whole $\p_+SM$. However, in view of the discussions in Remark \ref{generalize 1} and \ref{generalize 2} of the linear problem, it is reasonable to only consider the scattering data on $\supp \rho \cap \p_+SM$, which now becomes a partial data problem. We do not expand the details here.
\end{rem}


\appendix

\section{A complex to real reduction}\label{complex to real}

Given $\mathcal A=\mathcal A^r+i \mathcal A^i$ with real $\mathcal A^r$ and  $\mathcal A^i$, consider the solution $W$ to the transport equation
$$XW=W\mathcal A, \quad W|_{\p_+SM}=\id.$$ 
Similar to $\A$, we write $W$ as $W=W^r+iW^i$, then given $h=[\alpha,f]=h^r+i h^i$
$$I_\A h=\int (W^r+iW^i)(h^r+ih^i)\,dt=\int (W^rh^r-W^ih^i)\,dt+i\int (W^rh^i+W^ih^r)\,dt.$$
Thus we can separate the real and imaginary parts of $I_\A h$ and rewrite the X-ray transform as 
$$I_{\mathcal W}\begin{pmatrix} h^r \\ h^i\end{pmatrix}=\int \mathcal W \begin{pmatrix}h^r \\ h^i \end{pmatrix} \, dt$$
with 
$$\mathcal W=\begin{pmatrix} W^r & -W^i \\[.5em] W^i & W^r \end{pmatrix}.$$

\begin{lemma}
The matrix weight $\mathcal W$ is invertible if and only if $W$ is invertible.
\end{lemma}

\begin{proof}
By the linear algebra
$$det\, \mathcal W=det\, (W^r-iW^i)\,det \,(W^r+iW^i)=det\,\overline W\,det\, W=|det\,W|^2,$$
this implies the lemma.
\end{proof}

It is easy to check that 
$$X\mathcal W=\mathcal W\hat \A\quad \mbox{with}\quad \hat\A:=\begin{pmatrix} \mathcal A^r & -\mathcal A^i \\ \mathcal A^i & \mathcal A^r \end{pmatrix},\quad \mathcal W|_{\p_+SM}=\id.$$
The natural elements in the kernel of $I_{\mathcal W}$ are
$$(d+\hat{\mathcal A})\begin{pmatrix} p^r \\ p^i \end{pmatrix}=\begin{pmatrix} h^r \\ h^i \end{pmatrix}$$
with $p^r|_{\p M}=p^i|_{\p M}=0$, which is equivalent to that
$$(d+\mathcal A)(p^r+ip^i)=h,$$
the natural elements in the kernel of $I_\A$ by defining $p=p^r+ip^i$. Based on above discussions, it is easy to see that $I_\A$ and $I_{\mathcal W}$ are equivalent. While it is purely in the real category when considering $I_{\mathcal W}$.

\section{An orthogonal decomposition of pairs $h=[\alpha,f]$}\label{decomposition}

The following theorem is an analogue of \cite[Theorem 3.3.2]{Sh00} on the decomposition of symmetric tensor fields.

\begin{thm}\label{decomposition thm}
Let $M$ be a compact manifold with boundary. Let $k\geq 0$, for every pair $h\in {\bf H}^k(M)$, there exist unique $h^s\in \bfH^k(M)$ and $p\in H^{k+1}(M)$ such that 
$$h=h^s+d_\A p,\quad \delta_\A h^s=0,\quad p|_{\p M}=0.$$
\end{thm}

\begin{proof}
Consider the following Dirichlet boundary value problem
\begin{equation}\label{Dirichlet problem}
\delta_\A d_\A u=\varphi,\quad u|_{\p M}=0.
\end{equation} 
If for any $\varphi\in H^{k-1}(M)$, \eqref{Dirichlet problem} has a unique solution $u\in H^{k+1}(M)$, then we take $\varphi=\delta_\A h$, and $h^s=h-d_\A u$ to prove the claim. So now the main task is to study the Dirichlet problem \eqref{Dirichlet problem}. 

First it is easy to see that $\sigma_p(\delta_\A d_\A)(x,\xi)=|\xi|^2$, so $\delta_\A d_\A$ is an elliptic differential operator of order $2$, and the Dirichlet boundary condition is coercive in this case. Thus we only need to show that the elliptic problem has trivial kernel and cokernel.

If $\delta_\A d_\A u=0$ and $u|_{\p M}=0$, by the ellipticity, $u$ is smooth and by Green's formula
$$0=(\delta_\A d_\A u,u)=\|d_\A u\|^2_{\bfL},$$
i.e. $d_\A u=0$.
Now for any $x_0\in M^o$, there exists a geodesic $\gamma:[0,T]\to M$ connecting $x_0$ with the boundary $\p M$ at $y$ with $(\gamma(0),\dot\gamma(0))=(y,v_0)\in \p_+SM$. Notice that for $(x,v)\in SM$
$$0=W(x,v)d_\A u(x,v)=W(x,v)(Xu)(x,v)+(W\A)(x,v)u(x)=X(Wu)(x,v),$$
thus $\frac{\p}{\p t}W(\gamma(t),\dot\gamma(t))u(\gamma(t))$ is constant along any geodesic. We get that
$$W(x_0,\dot\gamma(T))u(x_0)=W(y,v_0)u(y).$$
Since $u(y)=0$ by assumption, and $W$ is invertible, we conclude that $u(x_0)=0$, which implies that $u\equiv 0$ on $M$.

To show the cokernel is trivial, it is enough to pick an arbitrary $\varphi\in C^{\infty}(M)$ which is orthogonal to the image $\{\delta_\A d_\A u: u|_{\p M}=0\}$. Then we have that for any $u\in C^{\infty}_c(M^o)$ (so $d_\A u\in C^{\infty}_c(M^o)$ too)
$$0=(\varphi,\delta_\A d_\A u)=(d_\A \varphi, d_\A u)=(\delta_\A d_\A \varphi, u).$$
Thus $\delta_\A d_\A \varphi$ is orthogonal to any $u\in C^{\infty}_c(M^o)$, which implies that $\delta_\A d_\A \varphi=0$. Now given arbitrary $\psi\in C^{\infty}(\p M)$, one can easily find some $u\in C^{\infty}(M)$ with $u|_{\p M}=0$ and $du(\nu)|_{\p M}=\psi$. Therefore by Green's formula
\begin{align*}
0&=(\delta_\A d_\A \varphi, u)_M=(d_\A \varphi, d_\A u)_M\\
&=(\varphi, \delta_\A d_\A u)_M+(\varphi,(d+A)u(\nu))_{\p M}=(\varphi,du(\nu))_{\p M},
\end{align*}
since $\varphi$ is in the cokernel. Then $(\varphi,\psi)_{\p M}=0$ for any $\psi\in C^{\infty}(\p M)$ which implies that the trace of $\varphi$ on the boundary is zero. Together with the fact $\delta_A d_\A \varphi=0$, we conclude that $\varphi=0$ and the theorem is proved.
\end{proof}


\end{document}